\documentclass{amsart}

\title[Non-Asymptotic Concentration of Magnetization in the Curie-Weiss Model ]{Non-Asymptotic Concentration of Magnetization in the Curie-Weiss Model at Subcritical Temperatures}

\newcommand{\bS}{{\bf S}}               

\newcommand{\Real}{\mathbb R}

\newcommand{\al}{\alpha}                
   
\newcommand{\ra}{\rightarrow}           

\newtheorem{thm}{Theorem}
\newtheorem{lem}[thm]{Lemma}
\newtheorem{pro}[thm]{Proposition}


\author{Yingdong Lu\\ IBM T.J.Watson Research Center}

\begin{document}

\maketitle
\begin{abstract}
In this short paper, we obtain non-asymptotic concentration results for magnetization of the Curie-Weiss model at subcritical temperatures, 
which leads to a diffusion limit theorem of the scaled and centered magnetization driven by a Metropolis-Hasting algorithm. 
These results are complementary to results at supercritical and critical temperatures in~\cite{bierkens2017}. 
\end{abstract}

\section{Introduction}
\label{sec:intro}

Curie-Weiss model is among the simplest ferromagnetic models in statistical mechanics for quantifying phase transition, see detailed explanations in Chapter IV of~\cite{ellis2006entropy}. 
Mathematically,  Curie-Weiss model is defined by the following probability measure on $\bS:=\{-1,1\}^n$,
\begin{align}
	\label{eqn:CW_defn}
	\pi^n(x)=\exp[-\beta H^n(x)]/Z_n, \quad x\in \bS,
\end{align}
with
\begin{align*}
	H^n(x) = \frac{1}{2n} \sum_{i,j=1}^n x_i x_j -h\sum_{i=1}^n x_i,
\end{align*}
where $Z_n$ is the normalizing constant (also known as the \emph{partition function}), $\beta>0$ represents the reciprocal of the absolute temperature, 
and $h\in \Real$ is the strength of external applied magnetic field. For each state $x\in \bS$, a key quantity \emph{magnetization}, denoted as $m^n(x)$, is defined as $m^n(x) := \sum_{i=1}^nx_i$.  
Then, the Hamiltonian function $H^n(x)$ can be expressed as,
\begin{align*}
	H^n(x) = -n \left[\frac12(m^n(x))^2+ hm^n(x)\right].
\end{align*}

A common practice in computing the partition function, thus the entire probability distribution, is to use Markov chain Monte Carlo (MCMC) methods, and Metropolis-Hastings algorithm is one of such methods that is found to be successful in many cases. In~\cite{bierkens2017},  it is shown that magnetization, under proper a centering and scaling transformation, the Markov chain generated from the Metropolis-Hastings algorithm has a diffusion limit for the critical ($\beta=1$ and $h=0$) and supercritical ($\beta<1$) temperatures. The derivation depends on results non-asymptotic concentration derived in these two cases. In the paper, we are able to obtain a similar non-asymptotic concentration at subcritical temperatures (($\beta<1$ and $h\neq 0$), thus obtain a similar diffusion limit.

The rest of the paper is organized as follows, in Sec. \ref{sec:results}, we will present concepts and results related to diffusion approximations of centered and scaled megnetization; in Sec.\ref{sec:concentration}, the key results in non-asymptotic concentration are presented.

\section{Diffusion approximations to Magnetization}
\label{sec:results}

\subsection{Concentration of meanetization}
It is known, see e.g. section IV.4 in~\cite{ellis2006entropy}, that the megnetization $m^n$ concentrates around $m_0$, the unique minimum point of the following quantity, 
\begin{align}
	\label{eqn:rate_function}
	-\left(\frac12\beta m^2 +\beta h m \right) + \frac{1-m}{2} \log (1-m) + \frac{1+m}{2}\log (1+m), \quad m\in (-1,1).
\end{align}
In~\cite{ellis2006entropy}, the analysis of concentration and identification of $m_0$ are obtained through a large deviations principle analysis, see also~\cite{bovier_2006} for a different and simplified derivation.
Furthermore, $m_0$ satisfies the following Curie-Weiss equation
\begin{align}
	\label{eqn:curie-weiss}
	\beta m+\beta h =\frac12 \log \frac{1+m}{1-m},
\end{align}
which can also be written as $m =\tanh (\beta (m+h))$. Quantitatively, the concentration can be elaborated by the following key result from~\cite{Chatterjee2007},
\begin{pro} [~\cite{Chatterjee2007}]
	For all $\beta\ge 0$, $h\in \Real$ and $t\ge 0$,
	\begin{align}
		\label{eqn:concentration}
		\pi^n \left(|m^n-\tanh(\beta(m^n+h))|\ge \frac{\beta}{n} +\frac{t}{\sqrt{n}}\right) \le 2 \exp \left(-\frac{t^2}{4(1+\beta)}\right).
	\end{align}
\end{pro}


\subsection{Metropolis-Hastings Algorithm}

         A typical implementation of the Metropolis-Hasting algorithm for sampling from a distribution known up to the partition function uses a random walk to generate 
         a proposal, then accept or reject the proposal according the ratio of the density function of the target distribution, and this does not require the calculation of the partition function. 
         
         The random walk on $\bS$ under consideration randomly picks (with probability $1/n$ each) one coordinate and flip its sign. Therefore, the transition probability takes the form of
         \begin{align*}
         	P(x,y)=
         	\left\{ \begin{array}{cc}\frac{1}{n} & y\in R(x),
         		\\ 0& \hbox{otherwise} \end{array} \right.
         \end{align*}
         where the set $R(x) = \{ y\in \bS: \hbox{$y_k=-x_k$ for some $k\in [n]$, and $y_i=-x_i$ for all $i\neq k$}\}$. To ensure reversibility hence convergence, the Metropolis-Hastings
         step accepts a proposed move generated by the random walk with probability
         \begin{align*}
         	\min\left\{ 1, \frac{\pi^n(y)P(y,x)}{\pi^n(x)P(x,y)}\right\}=\min\left\{ 1, \frac{\pi^n(y)}{\pi^n(x)}\right\}.
         \end{align*}
         The equality is due the fact that $P(y,x)=P(x,y)=\frac{1}{n}$ in the case that $y$ is a proposed move to $x$. 
         \subsection{Centered and Scaled Magnetization}
         
         The key quantity that will be studied is the following centered and scaled magnetization,
         \begin{align*}
         	\eta^n (x) := n^{1/2} [m^n(x)-m_0], \quad x\in \bS.
         \end{align*}
         It was termed shifted and renormalized magnetization in~\cite{bierkens2017}. Following Metropolis-Hastings algorithm, the scaled Markov chain, denoted by $X^n$, has the transition probability,
         \begin{align*}
         	P^n(\eta, \eta\pm 2n^{-\frac12}) = Q^n(\eta,  \eta\pm 2n^{-\frac12})(1\wedge \exp\{\beta[\Phi^n(\eta)-\Phi^n( \eta\pm 2n^{\frac12})]\}),
         \end{align*}
         with
         \begin{align*}
         	Q^n(\eta,  \eta\pm 2n^{-\frac12}):= \frac12(1\mp (m_0+ n^{-\frac12} \eta)), \quad \Phi^n(\eta)=-\frac12 \eta^2-n^\frac12(m_0+h)\eta.
         \end{align*}
         Let $Y^n$ represent the stationary continuous time Markov chain that jumps at rate $n$ with transition probability $P^n$, and its stationary distribution $\mu^n \propto \exp(-\beta \Phi^n(\eta))$. 
         
         \subsection{Diffusion Approximation Theorem}
         
         In~\cite{bierkens2017}, analysis has been provided for the diffusion limit of the centered and scaled magnetization, as $n$ approaches infinity in the critical ($\beta=1$ and $h=0$) and supercritical ($\beta<1$) temperature. 
         In this note, we extend their result to the subcritical phase $\beta>1$ and $h\neq 0$, and establish the following result. 
         \begin{thm}
         	\label{thm:main}
         	Suppose $\beta>1$ and $h\neq 0$, $Y^n$ jumps at rate $n$, then $Y^n$ converge weakly in $D([0,\infty), \Real)$  to $Y$, where $Y$ is the stationary Ornstein-Uhlenbeck process defined by
         	\begin{align}
         		\label{eqn:limiting-diffusion}
         		dY(t) = 2\ell(h, \beta) Y(y) dt + \sigma(h, \beta) dB(t),
         	\end{align}
         	with \begin{align*}
         		\sigma(h, \beta)=2\sqrt{1-|m_0(h, \beta)|}, \quad \ell(h, \beta)=\frac{1}{1+|m_0(h, \beta)|}-\beta(1-|m_0(h, \beta)|).
         	\end{align*}
         \end{thm}


         The key to proving Theorem \ref{thm:main} is to establish non-asymptotic concentration results, which will be carried out in the next section.

\section{Non-Asymptotic Concentration Inequality for Sub-Critical Phase}
\label{sec:concentration}

When $\beta>1$, $h>0$ (The case of $\beta>1$, $h<0$ can be dealt similar by symmetry), it is known, see e.g. Chapter IV in~\cite{ellis2006entropy}, there are potentially three roots to the Curie-Weiss equation \eqref{eqn:curie-weiss}, with $m_0>0$ being the largest, and two other roots being negative. It was pointed in both ~\cite[Chapter VI]{ellis2006entropy} and ~\cite{bierkens2017}, in the case of $h\neq 0$, the other two roots are not global minimum. 
The following lemma decides the order of derivatives of the two sides of \eqref{eqn:curie-weiss} at $m_0$. 

         \begin{lem}
         	\label{lem:slope}
         	$\beta < \frac{1}{1-m_0^2}$ when $\beta >1$ and $h>0$.
         \end{lem}
         \begin{proof}
         	Define $m_1:=\sqrt{(\beta -1)/\beta}$, hence $m_1\in [0,1)$ and $\beta = \frac{1}{1-m_1^2}$.  Note that $\beta$ is the slope of the left hand side of the Curie-Weiss equation \eqref{eqn:curie-weiss}, which is in an affine form.  $\frac{1}{1-m^2}$ is the derivative of the right hand side, and it is increasing in $m$. $m_1$ is the point where these two derivatives are the same. When $m\ge m_1$, the derivative of the left hand side will remain to be $\beta$, on the other hand, the derivative of the right hand side will increase in $m$. In the following, we will show that $\beta m_1+\beta h > \frac12\log \frac{1+m_1}{1-m_1}$. Therefore, the equality of the two sides must happen after $m_1$, i.e., $m_1<m_0$, and the result follows. 
         	
         	At $m_1$, write both sides of the Curie-Weiss equation \eqref{eqn:curie-weiss}  as functions of $\beta$. We have, the left hand side takes the form of $f_L(\beta)=\sqrt{(\beta -1)\beta}+h\beta$, and the right hand side $f_R(\beta)=\log(\sqrt{\beta}+\sqrt{\beta-1})$. Hence, $f'_L(\beta)=\frac{2\beta-1}{2\sqrt{(\beta -1)\beta}}+h$, and  $f'_R(\beta)=\frac{1}{2\sqrt{(\beta -1)\beta}}$. So, $f'_L(\beta)> f'_R\beta)$ for $\beta\ge 1$ and $h>0$. Furthermore, $f_L(1)=h >0=f_R(1)$. Thus we conclude $f_L(\beta)>f_R(\beta)$, i.e. $\beta m_1+\beta h > \frac12\log \frac{1+m_1}{1-m_1}$.
         \end{proof}

         
         \begin{lem}
         	\label{lem:right_side}
         	There exist $\iota_0>0$ and $M_1, M_2\in[0,1]$ satisfying  $0\le M_1< m_0$ and $m_0<M_2\le 1$, such that
         
         	\begin{align}
         	\label{eqn:left_side}
         	\beta m+\beta h \ge \frac12 \log \frac{1+m'}{1-m'},
         \end{align}	
         	holds for all $M_1\le m\le m_0$, with $m'=m-\iota_0 (m-m_0)$. Meanwhile,
         	\begin{align}
         		\label{eqn:right_side}
         		\beta m+\beta h \le\frac12 \log \frac{1+m'}{1-m'},
         	\end{align}	
         	holds for all $m_0\le m\le M_2$.
         \end{lem}
         \begin{proof} Consider a bivariate  function $K(m, \iota) = \frac12 \log \frac{1+m-\iota (m-m_0)}{1-m-\iota (m-m_0)}-(\beta m+\beta h )$. It is easy to see that $K(m_0, 0)$=0. Meanwhile, we have,  $\frac{\partial}{\partial m}K(m_0, 0) =\frac{1}{1-m_0^2}-\beta>0$ from Lemma \ref{lem:slope}. Therefore, there exists a neighborhood of $(m_0, 0)$ within which $\frac{\partial}{\partial m}K(m, \iota)>0$.  Pick an $\iota_0>0$, as well as $M_1, M_2\in[0,1]$ satisfying  $0\le M_1< m_0$ and $m_0<M_2\le 1$, such that, $\frac{\partial}{\partial m}K(m, \iota_0)>0$ for $m\in [M_1, M_2]$. Apparently, $K(m, \iota_0)\le 0$ when $m\le m_0$ and $K(m, \iota_0)\ge 0$ when $m\ge m_0$ , and they correspond to the two inequalities  \eqref{eqn:left_side} and \eqref{eqn:right_side}.
         \end{proof}
          Denote
         		$F^{n,\delta}:=\{\eta\in X^n: |\eta|\le n^\delta\}$.
         Then the above lemmas allow us to obtain the following result, which extends Lemma 5 in~\cite{bierkens2017} to subcritical temperature. 
         \begin{lem}
         	\label{lem:outsideF}
         	Let $0<\delta<\frac12$. For $\beta>1$, $h\neq 0$, and any $\al>0$, we have, 
         	\begin{align*}
         		\lim_{n\ra \infty} n^\al \pi^n(\eta^n(x) \notin F^{n,\delta}) =0.
         	\end{align*}
         \end{lem}
         \begin{proof}
         	Lemma \ref{lem:right_side} indicates that there exists $\iota_0>0$, such that inequalities \eqref{eqn:left_side} and \eqref{eqn:right_side} are satisfied,
         	or equivalently, $|m- tanh(\beta(m+h))| \ge \iota_0 |m-m_0|$ for $m \in[M_1, M_2]$. Therefore, for sufficiently large $n$, we have
         	\begin{align*}
         		\pi^n(|m-m_0| \ge n^{\delta-\frac12}) \le&\pi^n(|m- tanh(\beta(m+h))|  \ge \iota_0 n^{\delta-\frac12})
         		\\=& \pi^n\left(|m- tanh(\beta(m+h))|  \ge  \frac{\beta}{n}+\frac{t_n}{\sqrt{n}}\right)
         		\le 2\exp\left(-\frac{t_n^2}{4(1+\beta)}\right),
         	\end{align*}
         	with $$t_n=\left(\iota_0 n^{\delta-\frac12} -\frac{\beta}{n}\right)n^{\frac12}.$$
         	The result follows from the fact that $t_n\rightarrow \infty$ in the order of $n^\delta$ as $n\rightarrow \infty$. 
         \end{proof}
         
         The following lemma is known in~\cite{bierkens2017}. 
         \begin{lem}
         	\label{lem:moments}
         	Let $0<\delta<\frac12$. For $h\neq 0, \beta>0$. we have, for $r <\min(1, (k+1)(\frac12-\delta))$, 
		\begin{align*}
         		\lim_{n\rightarrow \infty} \sup_{\eta\in F^{n,\delta}} |nE_\eta^n[(Y-\eta)^2]-\ell(h,\beta)^2|=0,
		\end{align*}
         	\begin{align*}
         		\lim_{n\rightarrow \infty} \sup_{\eta\in F^{n,\delta}} |n E_\eta^n[(Y-\eta)^2]-\sigma(h,\beta)^2|=0.
         	\end{align*}
         	\begin{align*}
         		\lim_{n\rightarrow \infty} \sup_{\eta\in F^{n,\delta}} n E_\eta^n[|Y-\eta |^p]=0,
         	\end{align*}
	for any $p>2$.
         \end{lem}

         \begin{proof}[Proof of Theorem \ref{thm:main}]
         	The proof follows the basic arguments of that of Theorem 1 in~\cite{bierkens2017}, along with the non-asymptotic concentration result in Lemma \ref{lem:outsideF}, and we reproduce the key steps for completeness. The basic approach is to compare the action on function in the domain of the generators of $Y^n$ and 
         	the diffusion process \eqref{eqn:limiting-diffusion}. Suppose that $G^n$ is the generator of $Y^n$ and $G$ is the one for the diffusion. Then, we know that,
         	\begin{align*}
         		G^n\phi(\eta)=n[P^n\phi(\eta)-\phi(\eta)],
         	\end{align*}
         and 
         \begin{align*}
         	G\phi(\eta)=-2\ell(h,\beta)\eta \frac{d\phi}{d\eta}+\frac12\sigma^2(h,\beta) \frac{d^2\phi}{d\eta^2},
         \end{align*}
     for $\phi\in D(G):=\{ \phi\in C_0^2(\Real):\eta\mapsto \eta\phi'(\eta)\in C_0(\Real)\}$. We have,

         	\begin{align*}
         		&\sup_{\eta\in F^{n, \delta}} |G^n \phi(\eta) -G\phi(\eta) |\\
         		\le & \sup_{\eta\in F^{n, \delta}}\Big| n E^n_\eta [\phi(Y^n)-\phi(\eta)]-n E^n_\eta[\phi'(\eta)(Y^n-\eta) +\frac12\phi''(\eta)(Y^n-\eta)^2]   \Big|\\ \le & \frac{1}{6}\|\phi^3\|_\infty E^n_\eta[|Y^n-\eta|^3]
         		\rightarrow 0, \quad \hbox { as } n\rightarrow \infty,
         	\end{align*}
         	where the first inequality follows from the definitions and second inequality follows from Taylor expansion and the convergence is the consequence of estimations of the first, second and higher order terms in Lemma \ref{lem:moments}. And 
         	\begin{align*}
         		P^n(Y^n \notin F^{n, \delta} \hbox{ for some } 0\le t\le T) \le n \pi^n (\eta^n(x) \notin F^{n, \delta}) \rightarrow 0\quad \hbox { as } n\rightarrow \infty,
         	\end{align*}
         	follows from Lemma \ref{lem:outsideF}. Then the desired result follows from  Corollary 4.8.7 in~\cite{ethier2009markov}.
         \end{proof}
         

         
         
         
         \bibliographystyle{abbrvnat}
         \bibliography{Lu}

     \end{document}